\documentclass[12pt]{amsart}

\usepackage{amsmath,amsthm,amssymb}
\usepackage{bbm}
\usepackage{mathrsfs}
\usepackage{enumerate}
\usepackage{graphicx}
\usepackage{caption} 
\usepackage{subcaption} 
\usepackage{thmtools}
\usepackage{thm-restate}

\newcommand{\del}{\partial }

\newtheorem{thm}{Theorem}

\newtheorem{prop}[thm]{Proposition}

\newtheorem{corr}[thm]{Corollary}

\theoremstyle{definition}
\newtheorem{ex}[thm]{Example}
\newtheorem{defn}[thm]{Definition}
\theoremstyle{remark}
\newtheorem{rem}[thm]{Remark}
\begin{document}

\title{Trisections of 3-Manifolds}
\author{Dale Koenig}
\date{}

\begin{abstract}
We define a trisection of a closed, orientable three dimensional manifold into three handlebodies, and a notion of stabilization for these trisections.  Several examples of trisections are described in detail.  We define the trisection genus $t(M)$ of a 3-manifold, and relate it to the Heegaard genus $g(M)$, showing that $t(M) \le g(M) \le 2t(M)$.  We show moreover that the bound $g(M) \le 2t(M)$ is tight. We define stabilizations of trisections and show that all trisections of a 3-manifold are stably equivalent, providing an analogue of the Reidemeister-Singer theorem for trisections.  We conclude by showing that there exist complicated trisections of $S^3$.
\end{abstract}

\maketitle

\section{Introduction}

A Heegaard splitting of a closed, orientable 3-manifold can be thought of as a ``bisection'' of the 3-manifold into two handlebodies.  Gay and Kirby \cite{GayKirby1} introduced trisections of smooth, orientable 4-manifolds to create an analogous construction in the higher dimension, defining the trisection genus of a 4-manifold and proving that all trisections of a 4-manifold are stably equivalent.  In this paper we consider these ideas back in the third dimension.  Decompositions of non-orientable 3-manifolds into three orientable handlebodies have been analyzed by Gomez-Larra{\~n}aga, Heil, and N{\'u}{\~n}ez\cite{gomezheilnunez1}\cite{gomezheilnunez2} who also defined the notion of the trigenus of a nonorientable 3-manifold.  Gomez-Larra{\~n}aga also investigated which orientable 3-manifolds can be decomposed into three tori \cite{gomez1}.  Coffey and Rubinstein have looked at orientable 3-manifolds formed by gluing three handlebodies in a sufficiently complicated way \cite{CoffeyRubinstein1}.  In this paper we will look at decompositions of orientable three manifolds into three handlebodies with connected pairwise intersections.  This condition allows us to draw connections with the field of Heegaard splittings. 

We will firsta trisection of a closed, orientable 3-manifolds and a notion of stabilization on these trisections.  We then investigate several examples in detail, showing some surprising trisections.  We define the trisection genus of a 3-manifold, and relate it to the Heegaard genus of the manifold.  We analyze the behavior of trisection genus under connect sum, showing that if $M$ is the connect sum of two manifolds of Heegaard genus $g$, $M$ has trisection genus equal to half its Heegaard genus.  We then prove the main theorem of the paper, showing that with one trivial exception, all trisections of a closed, orientable 3-manifold $M$ can be made equivalent by stabilization.  

We begin with the definition of a trisection.  Let $M$ be a closed, orientable 3-manifold.

\begin{defn}
A $(h_1,h_2,h_3;b)$-\emph{trisection} of $M$ is a quadruple \linebreak $(H_1,H_2,H_3;B)$ such that
\begin{itemize}
\item $M = H_1 \cup H_2 \cup H_3$
\item $H_i$ is a handlebody of genus $h_i$
\item Each $S_{ij} = H_i \cap H_j$ is a compact connected surface with boundary $K$
\item $B = H_1 \cap H_2 \cap H_3$ is a $b$-component link
\end{itemize}
\end{defn}

If $h_1 = h_2 = h_3=h$, the trisection is \emph{balanced}, and we call it an $(h;b)$-trisection.  Otherwise, it is \emph{unbalanced}.  If $(H_1,H_2,H_3;B)$ is a balanced $(h;b)$-trisection, we define the \emph{genus} of the trisection to be $h$.  Note that, in contrast to trisections of 4-manifolds where the genus refers to the complexity of the triple intersection, here it refers to the genus of the handlebodies $H_1,H_2,H_3$.  The simplest trisection is the trisection of $S_3$ into three balls, with each pair of balls intersecting in a disk.  We refer to this as the \emph{trivial} trisection of $S^3$.  Many more examples of trisections will be covered in section 2.

\begin{defn} Let $i,j,k$ be the indices 1,2,3 in any order.  Suppose that $S_{jk}$ is not a disk, and let $\alpha$ be a nonseparating arc in $S_{jk}$.  Define a new trisection $(H_1',H_2',H_3';B')$ by
\begin{align*}
H_i' &= \overline{H_i \cup N(\alpha)} \\
H_j' &= \overline{H_j - H_j \cap N(\alpha)} \\
H_k' &= \overline{H_k - H_k \cap N(\alpha)} \\
B' &= H_i' \cap H_j' \cap H_k'
\end{align*}
This results in a new trisection of $M$ where $h_i$ is increased by 1, and $b$ is changed by $\pm 1$.  We call this operation a \emph{stabilization}.  This operation depends on the choice of $\alpha$, so it is not generally unique even after fixing a choice of $i,j,k$.
\end{defn}

Two trisections $(H_1,H_2,H_3;B)$ and $(H_1^*,H_2^*,H_3^*;B^*)$ are \emph{isotopic} if there is an isotopy of $M$ taking each $H_i$ to the corresponding $H_i^*$ and taking $B$ to $B^*$.  We say that one trisection of $M$ is a \emph{stabilization} of another if it can be obtained by some sequence of stabilizations, up to isotopy.  Notice that a relabelling of the handlebodies does not necessarily produce an isotopic trisection.  So, for example, $(H_1,H_2,H_3;B)$ and $(H_2,H_1,H_3;B)$ may be distinct trisections.  In some settings, the order of the handlebodies may be unimportant.  In the examples in the following section, we provide one possible order of the handlebodies, and implicitly treat all reorderings as part of the same class of examples.

 We can now state the main theorem.

\begin{restatable}{thm}{themainthm}
\label{mainthm}
Let $M$ be a closed, orientable 3-manifold with two trisections.  If $M = S^3$, assume that neither of the two trisections is the trivial trisection into three balls.  Then there exists a third trisection isotopic to a stabilization of each of the original two trisections.
\end{restatable}
In the following section, we will discuss several examples of trisections, and methods to obtain interesting trisections for many classes of manifolds.  Section 3 we discuss how to get a balanced trisection from an unbalanced one.  In section 4 we will define the trisection genus of a manifold and relate the trisection genus and Heegaard genus of $M$.  Section 5 will present the proof of Theorem \ref{mainthm}.  In the final section we will prove that there is no reasonable analogue of Waldhausen's theorem for trisections.

Special thanks to Abby Thompson for many useful discussions.


\section{Examples}
We begin with a few ways to get trisections of any 3-manifold, and then present some more interesting trisections of specific classes of 3-manifolds.  The order of the handlebodies is unimportant for producing examples, so we will use whatever order is convenient, usually ordering from largest to smallest genus.

\begin{ex}
\label{trifromheeg}
Let $M$ be any closed orientable 3-manifold, and let $V \cup_\Sigma W$  be a genus $g$ Heegaard splitting of $M$.  Let $D$ be a disk in $\Sigma$.  Define $H_3$ to be a regular neighborhood of $D$, and let $H_1 = \overline{V - V \cap H_3}$, $H_2 = \overline{W - W \cap H_3}$.  Then this defines a $(g,g,0;1)$-trisection of $M$.  We can stabilize $H_3$ $g$ times to produce a balanced genus $g$ trisection of $M$.  See Figure \ref{trifromheegpic}.  The trivial $(0,0,0;1)$-trisection of $S^3$ is a special case of this construction.
\end{ex}

\begin{figure}[ht]
\centering
\includegraphics[scale=.45]{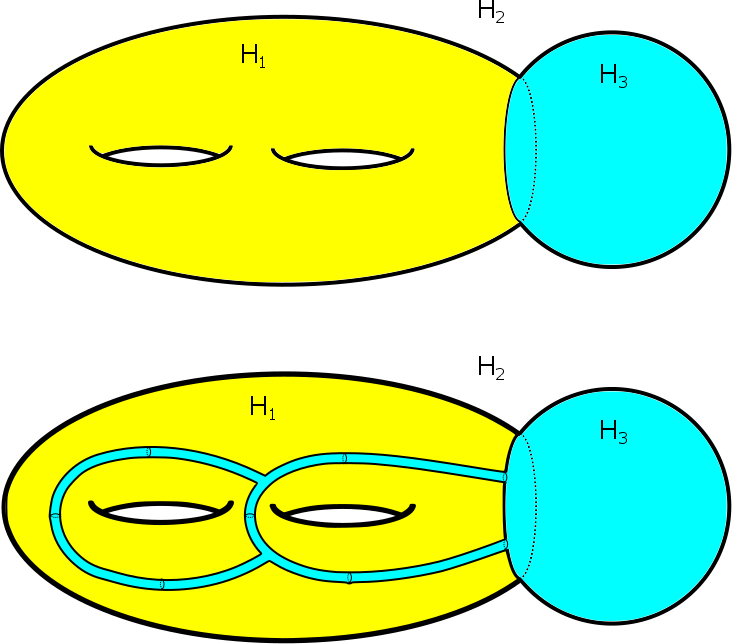}
\caption{A $(2,2,0,1)$ trisection is constructed from a genus-2 trisection of $S^3$ by the construction of Example \ref{trifromheeg}.  A balanced trisection can be obtained by stabilizing $H_3$ twice.}
\label{trifromheegpic}
\end{figure}

\begin{ex}
Suppose $(K,\phi)$ is an open book decomposition of $M$ with binding circle $K$ and $\phi\colon M- K \to S^1$.  Then we can define 
\begin{align*}
&H_1 = K \cup \phi^{-1}([0,1/3]) \\
&H_2 = K \cup \phi^{-1}([1/3,2/3]) \\
&H_3 = K \cup \phi^{-1}([2/3,1])
\end{align*}
Since each $H_i$ is a thickening of a once punctured surface, each is indeed a handlebody.  This gives a $(2g,2g,2g;1)$-trisection of $M$ where $g$ is the genus of the fiber surface.

In fact, this is a special case of Example \ref{trifromheeg}.  If we set $V = K \cup \phi^{-1}([0,1/2])$ and $W = K \cup \phi^{-1}([1/2,1])$ we get a Heegaard splitting of $M$.  Applying the technique of Example \ref{trifromheeg} gives a $(g,g,0;1)$-trisection which can be stabilized to the $(g;1)$-trisection described in this example.  See Figure \ref{fiberedtrisection}.

 It is known that any two open book decompositions of $S^3$ are related by plumbing and deplumbing Hopf bands \cite{GirouxGoodman1}.  Hopf plumbings gives a different notion of stabilization from that used here, but it is worth noting that plumbing and deplumbing of Hopf bands is also not a unique operation.  Trisections of this form appear on the boundary of relative trisections of 4-manifolds as defined in \cite{GayKirby1} and \cite{CastroGayPinzon1}.
\end{ex}

\begin{figure}[ht]
\centering
\includegraphics[scale=.8]{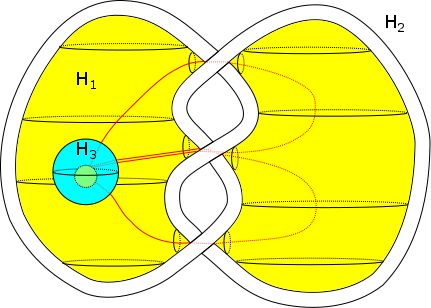}
\caption{We see a $(2,2,0;1)$ trisection where the first two handlebodies are neighborhoods of Seifert surfaces of the trefoil knot.  If we stabilize $H_3$ along the two red arcs, we get a trisection where each handlebody is a neighborhood of a Seifert surface, and the triple intersection curve $B$ is a trefoil.}
\label{fiberedtrisection}
\end{figure}

\begin{ex}
\label{trifromheeg2}
We can generalize the construction of Example \ref{trifromheeg} as follows.  Let $V \cup_\Sigma W$ be a genus $g$ Heegaard splitting of $M$.  Let $H_1 = W$.  Choose some disk properly embedded in $V$ that cuts it into two handlebodies $H_2$ and $H_3$ of genus $h$ and $g-h$ respectively, where $0 \le h \le g$.  Then this gives a $(g,h,g-h;1)$-trisection.  When $h=0$ or $h=g$ this construction reduces to the construction of Example $\ref{trifromheeg}$, possibly after relabelling the handlebodies.
\end{ex}

\begin{figure}[ht]
\centering
\includegraphics[scale=.35]{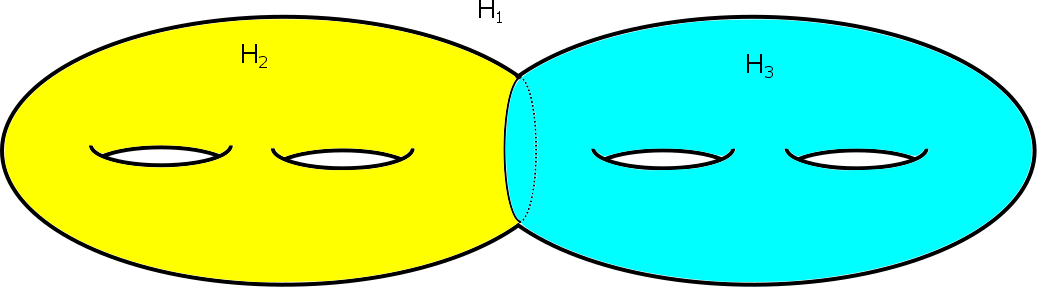}
\caption{Here we construct a $(4,2,2,1)$-trisection of $S^3$ using the technique of Example \ref{trifromheeg2}.  When $H_2$ or $H_3$ has genus 0 then we can relabel handlebodies and perform an isotopy to get the upper trisection in Figure \ref{trifromheegpic}.  Although the disk $H_2 \cap H_3$ here cuts the complement of $H_1$ into two standard handlebodies, it is possible that $H_2$ and $H_3$ are knotted.  An example of this is described in Example \ref{trifromknot}. }
\label{trifromheeg2pic}
\end{figure}

\begin{ex}
\label{trifromknot}
We describe a specific instance of the construction of Example \ref{trifromheeg2}.  Let $K$ be some knot in $S^3$.  Set $H_1 = \overline{N(K})$.  Let $D$ be a disk in $\del H_1$, and $\alpha_1,\dots,\alpha_m$ be a tunnel system for $K$, with the endpoints of each $\alpha_i$ lying in $D$.  Then we can set $H_2 = \overline{N(D \cup \bigcup_i \alpha_i)} $ and $H_3 = \overline{M - H_1 \cup H_2}$.  This gives a $(1,m,m+1;1)$ trisection of $S^3$.  Here $H_1 \cup H_2$ is a handlebody and $H_1 \cap H_2$ a disk.  See Figure \ref{treftrisection} for an example.
\end{ex}
\begin{figure}[ht]
\centering
\includegraphics[scale=.7]{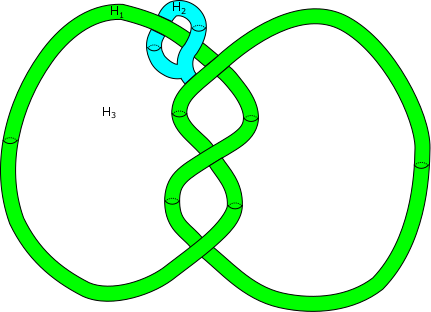}
\caption{A trisection constructed from a trefoil knot and a tunnel. }
\label{treftrisection}
\end{figure}

All examples so far have been stabilizations of the construction in Example \ref{trifromheeg2}.  To provide some different classes of examples, we present some trisections where all three handlebodies have genus lower than the Heegaard genus of the manifold.

\begin{ex}
Suppose $M$ is the connect sum of two 3-manifolds of Heegaard genus 1.  That is, each connect summand is either $S^1 \times S^2$ or a lens space.  $M$ has Heegaard genus 2 by \cite{Haken1}, so we can fix a genus 2 Heegaard splitting $X_1 \cup X_2$.  We produce a $(1,1,1;2)$-trisection of $M$.  See Figure \ref{s1s2tripic} for the case where $M$ is the connect sum of two copies of $S^1 \times S^2$.  We describe this case in detail.

Let $S$ be the reducing sphere splitting $M$ into the two copies of $S^1 \times S^2$. Begin by splitting $X_1$ along the disk $X_1 \cap S$, resulting in two genus 1-handlebodies $H_1$ and $H_2$.  Then both $\del H_1$ and $\del H_2$ contain essential curves $\alpha_1,\alpha_2$ respectively bounding disks in $X_2$.  Let $\gamma$ be an arc connecting $\alpha_1$ and $\alpha_2$ such that the intersection of $\gamma$ with  $\del H_1 \cap \del H_2$ is only a single point.  Let $\alpha$ now denote the result of performing a handle slide of $\alpha_1$ across $\alpha_2$ using the arc $\gamma$.  $\alpha$ intersects both $\del H_1$ and $\del H_2$ in a single arc, and bounds a disk $D$ in $X_2$. Therefore, we can isotope $H_2$ to add a neighbourhood of $D$.  We continue to call the result of this isotopy $H_2$.  $H_1$ and $H_2$ now intersect in an annulus essential in both $\del H_1$ and $\del H_2$, and $H_3 = M - (H_1 \cup H_2)$ is a genus 1-handlebody.  If follows that each intersection $S_{ij}$ is an annulus, and so all pairwise intersections are connected.  We therefore have a $(1,1,1;2)$-trisection as desired.  An identical argument can be applied when one or both of the connect summands are replaced with lens spaces.

\end{ex}
\begin{figure}[ht]
\centering
\includegraphics[scale=.4]{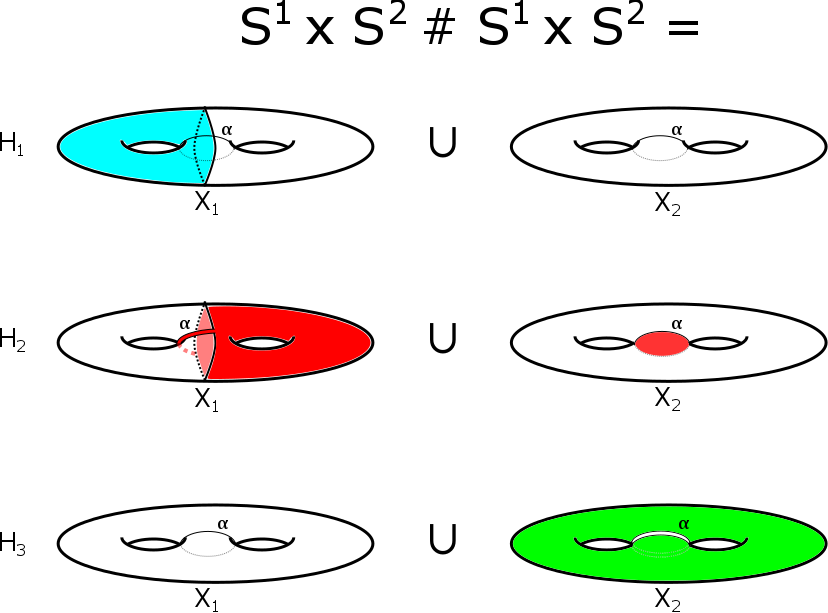}
\caption{A trisection of the connect sum of two copies of $S_1 \times S_2$, viewed as the union of two standard genus 2 handlebodies glued by the identity map on their boundary.  Each row shows how one of the three handlebodies lies in the manifold.}
\label{s1s2tripic}
\end{figure}

\begin{ex}
\label{connectsumtri}
Now consider a more general connect sum $M = M_1 \# M_2$ where both of $M_1,M_2$ have Heegaard genus $g$.  $M$ then has Heegaard genus $2g$ \cite{Haken1}.  Fix genus $g$ Heegaard splittings $(X_1,X_2,\Sigma)$ and $(X_1^*,X_2^*,\Sigma^*)$ for $M_1$ and $M_2$ respectively.  Let $\alpha_1,\cdots,\alpha_g$ (resp. \linebreak $\beta_1,\cdots,\beta_g$) be a collectively nonseparating set of $g$ disjoint curves on $\Sigma$ (resp. $\Sigma^*$) such that each curve bounds a disk in $X_2$ (resp. $X_2^*$).   $(X_1 \# X_1^*,X_2 \# X_2^*;\Sigma \# \Sigma^*)$ is a minimal genus Heegaard splitting for $M$.  Let $C$ be a curve splitting $\Sigma \# \Sigma^*$ into the punctured copies of $\Sigma$ and $\Sigma^*$.  We can apply a diffeomorphism to $\Sigma \# \Sigma^*$ that fixes $C$ and sends $\Sigma$ to $\Sigma$ and $\Sigma^*$ to $\Sigma^*$ in order to get a Heegaard diagram of the form shown in Figure \ref{standardsum}.  For each $i$, let $\gamma_i$ denote the result of sliding $\alpha_i$ across $\beta_i$ using connecting arcs intersecting $C$ once as shown in Figure \ref{standardsum}.  Let $D_1\dots D_g$ denote the meridian disks in $X_2 \# X_2^*$ bounded by the $\gamma_i$.

We can now define a trisection.  Let $H_1 = X_1$, so it is a genus $g$ handlebody.  Define $H_2$ to be the union of $X_1^*$ and the collection of all $N(D_i)$.  Removing these disks $N(D_i)$ from $X_2 \# X_2^*$ leaves another genus $g$ handlebody, which we define to be $X_3$.  $H_2$  as defined is isotopic to $X_1^*$, because we defined it by attaching disks that intersected $X_1^*$ in a single arc each.  Thus, attaching the disk $D_i$ is equivalent to isotoping $X_2$ to extend from the arc $D_i \cap \Sigma^*$ across the disk $D_i$ to the arc $D_i \cap \Sigma$.  We can also observe that each such attachment introduces a new curve component of $H_1 \cap H_2$, so the resulting trisection is a $(g,g,g;g+1)$ trisection.
 
\end{ex}
\begin{figure}[ht]
\centering
\includegraphics[scale=.4]{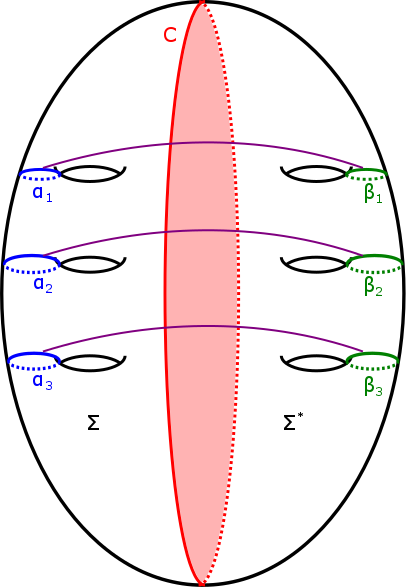}
\caption{The standard picture of what a connect sum Heegaard diagram looks like after forgetting about which curves on $\Sigma$ and $\Sigma^*$ bound disks in the inner handlebody.  The purple arcs are used to guide handle slides. }
\label{standardsum}
\end{figure}

\begin{ex}
Let $\Sigma$ be a closed orientable genus-$g$ surface.  Let $M$ be a surface bundle $\Sigma \times [0,1] / (x,0) \sim (\phi(x),1)$. $M$ has a genus $2g+1$ Heegaard splitting.  If the surface bundle is a product bundle $\Sigma \times S^1$ or if the translation distance of the monodromy map is sufficiently high relative to the genus of $\Sigma$, then it is known that the genus $2g+1$ Heegaard splitting is minimal \cite{Schultens1} \cite{BachmanSchleimer1}.  We produce a $(2g,g+1,g+1;b)$ trisection of $M$, where $b$ is either 1 or 3 depending on whether $g$ is even or odd.

\smallskip
\noindent\textit{Case 1.}\ \ 
$g$ is even.

See Figure \ref{surfbundtri}.  There is a curve $C \subset \Sigma$ cutting $\Sigma$ into two punctured genus $g/2$ surfaces.  Let $\alpha$ be a path in $\Sigma$ such that $\alpha(0)$ lies on $\phi(C)$ and $\alpha(1)$ lies on $C$.  Then the path  $P = \{\alpha(2t)\times t : 0 \le t \le 1/2\}$ is transverse to the fibers, so $\Sigma \times [0,1/2] - N(P)$ is homeomorphic to a thickened punctured genus g surface, and is therefore a genus 2g handlebody.  Let $H_1$ be this handlebody.  Now, $C \times [1/2,1]$ cuts $\Sigma \times [1/2,1]$ into two genus $g$ handlebodies $H_2,H_3$.  Split the tube $N(P)$ into two halves as in Figure \ref{surfbundtri}, and assign half to $H_2$ and half to $H_3$ so that they become genus $g+1$ handlebodies.  Note that performing twists to $N(P)$ will possibly produce non-isotopic trisections.  In the resulting trisection, $H_2 \cap H_3$ is a punctured torus, and each of $H_1 \cap H_2$ and $H_1 \cap H_3$ is the union of two punctured genus g surfaces connected by a band.
\begin{figure}[ht]
\centering
\includegraphics[scale=.45]{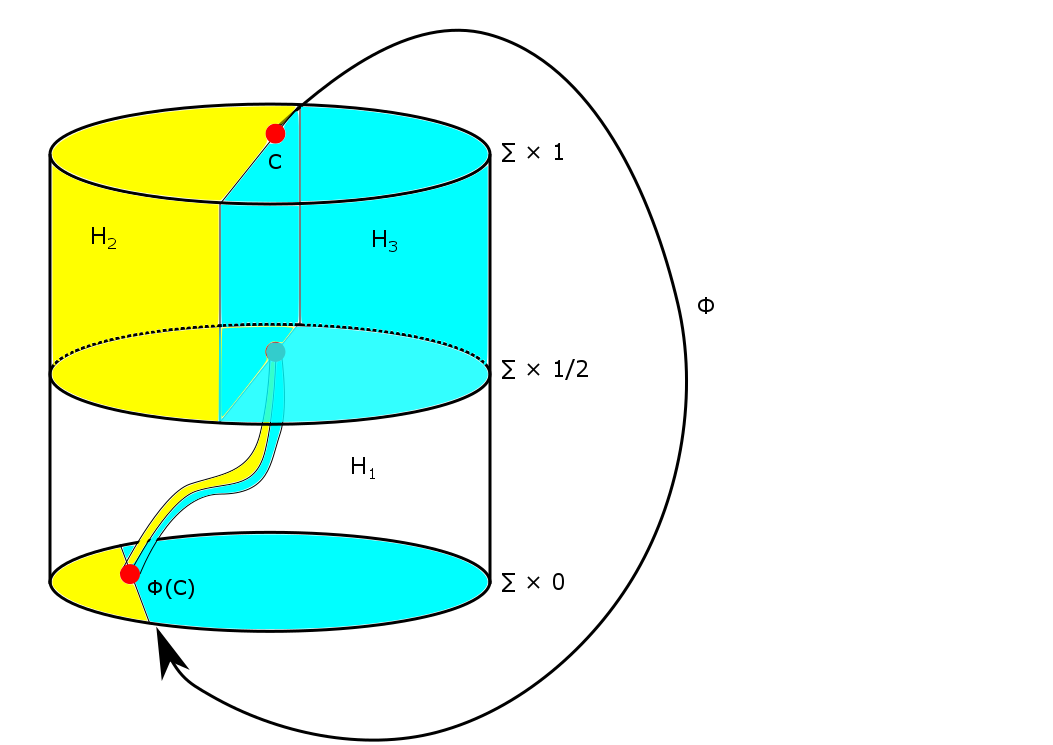}
\caption{A trisection of a surface bundle.}
\label{surfbundtri}
\end{figure}

\smallskip
\noindent\textit{Case 2.}\ \ 
$g$ is odd.

The idea is approximately the same.  Instead of $C$ we choose two curves $C_1,C_2$ cutting $\Sigma$ into two twice punctured genus $(g-1)/2$ surfaces.  Choose the path $\alpha$ to connect a point on $C_1$ to a point on $C_2$.  Everything else goes through as before, and we end up with a $(2g,g,g;3)$ trisection.  $H_2 \cap H_3$ is now a thrice punctured planar surface, and each of $H_1 \cap H_2$ and $H_1 \cap H_3$ is a thrice punctured genus $g-1$ surface.
\end{ex}


\section{Balancing Trisections}
We prove that an unbalanced trisection can be turned into a balanced trisection without increasing the genus of the largest handlebody.
\begin{prop}
\label{balanceprop}
 Let $(H_1,H_2,H_3;B)$ be an $(h_1,h_2,h_3;b)$ trisection of $M$.  Then there is a balanced $(h';b')$-trisection of $M$ that is a stabilization of $(H_1,H_2,H_3;B)$, where $h' = \max(h_1,h_2,h_3)$.  Additionally, we can ensure that $b' \le \max(b,2)$.  That is, either $b' \le b$ or $b' = 2$ and $b = 1$.
\end{prop}
\begin{proof}
Choose $i,j,k$ to be a permutation of $1,2,3$ such that $h_i \ge h_j \ge h_k$.  If $h_i = h_j = h_k$ then we are done.  Otherwise, we know that $h_i > h_k$.  Now, if $S_{ij}$ were a disk, then $H_i \cup H_j$ would be a genus $h_i + h_j$ handlebody with complement $H_k$.  This would give a Heegaard splitting and would imply that $h_i + h_j = h_k$, which contradicts $h_i > h_k$.  Therefore $S_{ij}$ is not a disk.  Hence there exists some nonseparating arc $\alpha$ properly embedded in $S_{ij}$.  Moreover, unless $b = 1$, we can choose $\alpha$ to have its endpoints lie on two distinct components of $B$.  Performing a stabilization with this choice of $\alpha$ gives a $(h_i,h_j,h_k+1; b')$-trisection, where $b'$ is either $b-1$ or $2$.  This operation does not increase the genus of any handlebody beyond $h_i$, and only increases $b$ if $b=1$.  Therefore, we can repeat the operation until we get a balanced trisection of genus $h_i$.
\end{proof}

We investigate the surfaces $S_{ij}$.  Again, let $i,j,k$ be some permutation of $1,2,3$.  We can compute the genus of the handlebody $H_i$ from $b$ and the genera $g(S_{ij})$ and $g(S_{ik})$ of $S_{ij}$ and $S_{ik}$ by the formula $h_i = g(S_{ij}) + g(S_{ik}) + b - 1$.  In a balanced trisection, $h_1 = h_2 = h_3$.  Comparing the formula for $h_1$ and $h_2$ we see that $g(S_{13}) = g(S_{23})$.  Similarly we can compare the formulas for $h_2$ and $h_3$ to see that $g(S_{12}) = g(S_{13})$.  Therefore, in a balanced trisection, all $g(S_{ij})$ are the same, and are equal to $\frac{h + 1 - b}{2}$.  Since this must be an integer, we also get the following:
\begin{rem}
\label{parityremark}
In a balanced $(h,b)$-trisection, $b$ and $h$ must have opposite parities.
\end{rem}

\section{Trisections, Heegaard Splittings, and Trisection Genus}

Just as the Heegaard genus $g(M)$ of a 3-manifold is defined as the smallest $g$ for which $M$ has a genus $g$ Heegaard splitting, we can define the trisection genus $t(M)$ to be the smallest $t$ for which $M$ has a balanced trisection of genus $t$.  Here we state some facts about trisection genus, and about how trisections relate to Heegaard splittings.

\begin{prop}
If $M$ is a closed orientable 3-manifold, its Heegaard genus $g(M)$ and trisection genus $t(M)$ are related by
$$t(M) \le g(M) \le 2t(M)$$
\end{prop}
\begin{proof}
First, note that by combining the construction of Example \ref{trifromheeg} or \ref{trifromheeg2} with Proposition \ref{balanceprop}, whenever $M$ has a Heegaard splitting of genus $g$ we can also construct balanced trisections of genus $g$.  It follows that $t(M) \le g(M)$.  We can also get a Heegaard splitting from a trisection $(H_1,H_2,H_3;B)$ as follows.  First choose one of the three handlebodies $H_i$, and let $j,k$ be the indices not chosen.  Choose a maximal set of nonseparating arcs in $S_{jk}$, and stabilize $H_i$ along each of these arcs in turn to get a new trisection $(H_1',H_2',H_3';B')$.  In this trisection, $S_{jk}'$ is now a disk, since if it were not then there would be some nonseparating arc in it, contradicting the maximality of our choice of arcs.  Therefore, $H_j' \cup H_k'$ is a handlebody.  If we started with a balanced trisection of genus $h$ then $h$ stabilizations were required to make $S_{jk}'$ a disk, so $g(H_i') = 2h$.  It follows that $(H_i',H_j' \cup H_k';\del H_i')$ is a genus $2h$ Heegaard splitting.  Applying this construction to a minimal genus balanced trisection of $M$, we conclude that $g(M) \le 2t(M)$.  
\end{proof}

Since the construction used in the previous proposition is quite useful, we set it aside as a definition.

\begin{defn}
Suppose $(H_1,H_2,H_3;B)$ is a trisection.  If we stabilize $H_i$ along a maximal set of arcs in $H_{jk}$, then we get a trisection $(H_1',H_2',H_3';B')$ where $H_j \cup H_k$ is a handlebody, so $(H_i, H_j \cup H_k;\del H_i)$ is a Heegaard splitting.  We call this the \emph{Heegaard splitting built from the trisection  $(H_1,H_2,H_3;B)$ by stabilizing $H_i$}.  If we do not care which $i$ was chosen, we just say that it is a Heegaard splitting built from the trisection.
\end{defn}

\begin{rem}
For a given trisection $(H_1,H_2,H_3;B)$, we do not know that, for example, the Heegaard splitting built by stabilizing $H_1$ and the Heegaard splitting built by stabilizing $H_2$ are isotopic.  However, after fixing a choice of handlebody $H_i$ we are stabilizing along a maximal set of arcs in $S_{jk}$.  Any two such maximal system of arcs in $S_{jk}$ are slide equivalent, so any two choices of arc systems will result in isotopic Heegaard splittings.  It follows that there are at most three isotopy classes of Heegaard splittings that can be built from a given trisection, one for each choice of handlebody $H_i$.
\end{rem}

It is natural to ask how strict these inequalities are.  We have already shown that the inequality $g(M) \le 2t(M)$ is the best general bound possible;  Example \ref{connectsumtri} demonstrates that if $M = M_1 \# M_2$ is a connect sum with $g(M_1) = g(M_2)$ then $t(M) = g(M_1)$.  However, it is known that $g(M) =2g(M_1)= 2t$ \cite{Haken1}.  This gives us
\begin{prop}
Suppose both $M_1$ and $M_2$ are closed orientable 3-manifolds with Heegaard genus $g$.  Let $M = M_1 \# M_2$.  Then $M$ has Heegaard genus $2g$ and trisection genus $g$.
\end{prop}
\begin{corr}
For each integer $t\ge 0$, there exists a 3-manifold with trisection genus $t$ and Heegaard genus $2t$.
\end{corr}

We can also ask whether for every $t \ge 0$ there exists a 3-manifold $M$ such that both the trisection and Heegaard genus are equal to $t$.  $S^3$ satisfies this for $t=0$, and any Lens space satisfies it for $t=1$.  The fact that there exist examples for $t=2$ follows from the following proposition relating Heegaard splittings built from trisections to Hempel distance \cite{Hempel1}.

\begin{prop}
\label{builtheegdist}
Suppose $(H_1,H_2,H_3;B)$ is a trisection of $M$ such that no $H_i$ has genus $g(H_i) = 0$.  Then any Heegaard splitting built from $(H_1,H_2,H_3;B)$ has distance at most 2.
\end{prop}
\begin{proof}
Suppose without loss of generalization that we build a Heegaard stabilization by stabilizing $H_1$.  Let the trisection achieved by stabilization be $(H_1',H_2',H_3';B)$ so that $(H_1',H_2' \cup H_3'; \del H_1')$ is the Heegaard splitting.  In order to demonstrate the distance bound we find a sequence of 3 curves $\alpha,\beta,\gamma$ on $\del H_1'$ such that $\alpha$ bounds a disk in $H_1'$ and $\gamma$ bounds a disk in $H_2' \cup H_3'$.  See Figure \ref{heegfromtri} for a picture of the case where each handlebody is genus 1, which easily generalizes to higher genus.

 Let $\alpha$ be a loop enclosing a cocore of one of the stabilizations that took $H_1$ to $H_1'$. Let $\beta$ be some curve in $H_1 \cap H_3 \subset \del H_1$, and note that the stabilization occured away from $\beta$, so we can treat $\beta$ as also lying in $H_1' \cap H_3'$. 

 Now to find $\gamma$, first choose any meridian disk $D$ of $H_2'$.  Since $H_2' \cap H_3'$ is a disk, we can isotope $D$ so that $\del D$ lies in $H_1' \cap H_2'$. 
  Set $\gamma$ to be $\del D$.  By construction $\alpha \subset H_1' - H_1 \cap H_3$, $\beta \subset H_1' \cap H_3'$, and $\gamma \subset H_1' \cap H_2'$.  Thus, $\alpha \cap \beta$ is empty, as is $\beta \cap \gamma$.  Moreover, $\alpha$ bounds a disk in $H_1'$ and $\gamma$ bounds a disk in the complement of $H_1'$, so this is indeed a distance 2 path. 
\end{proof}
\begin{figure}[ht]
\centering
\includegraphics[scale=.5]{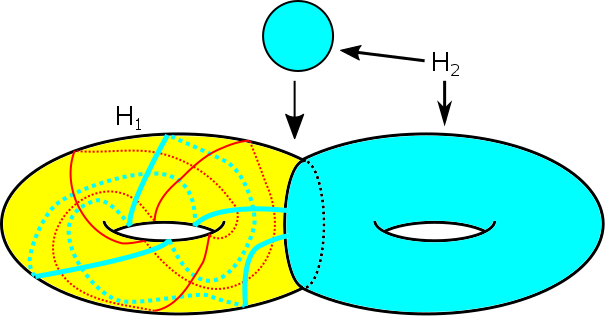}
\caption{Here we see a Heegaard splitting built from a balanced genus 1 trisection. $H_1$ is the yellow torus on the left, and $H_2$ the union of the blue torus, the blue band, and the blue disk attached in some way along the boundary.  $H_1'$ is obtained from $H_1$ by attaching an arc that intersects a meridian disk of the blue torus once.  Therefore, $H_1'$ is isotopic to the union of $H_1$ and the blue torus.   $\alpha$ is then a meridian curve of the blue torus, $\beta$ the red curve on the left, and $\gamma$ the attaching curve for the blue disk.}
\label{heegfromtri}
\end{figure}
\begin{corr}
There exist 3-manifolds with both Heegaard genus and trisection genus equal to 2.
\end{corr}
\begin{proof}
Suppose $M$ is a 3-manifold with a Heegaard splitting of genus $g=2$ and distance at least 5.  It is known that such manifolds exist by \cite{Hempel1}. It is known that when a Heegaard surface has distance $d > 2g$,  it represents the unique minimal genus Heegaard splitting \cite{ScharlemannTomova1}.  It follows that $M$ has no other genus 2 Heegaard splittings, and hence has no genus 2 Heegaard splitting of distance $\le 2$.  $M$ does have a trisection of genus 2 by Example \ref{trifromheeg}. If $M$ had a genus 1 trisection, it would have a distance 2 Heegaard splitting by Proposition \ref{builtheegdist}, which would be a contradiction.  Therefore, both the trisection genus and Heegaard genus of $M$ must be equal to 2. 
\end{proof}
This corollary can also be derived using the classification of genus 1 trisections by Gomez-Larra{\`n}aga \cite{gomez1}, since any 3-manifold not in his list that has a genus 2 Heegaard splitting necessarily also has trisection genus 2.

 It would be interesting to know whether there exist higher genus examples with trisection genus equal to their Heegaard genus.


\section{The Stabilization Theorem}
Before proving the theorem, we need one more definition.
\begin{defn}
Suppose $(H_1,H_2,H_3;B)$ is a trisection of $M$ such that $(H_1,H_2 \cup H_3; \del H_1)$ is a Heegaard splitting.  Suppose moreover that there exists a disk $D$ properly embedded in $H_1$ such that $\del D$ consists of a nonseparating arc in $S_{12}$ and a nonseparating arc in $S_{13}$.  We can stabilize $H_2$ and then $H_1$ as in Figure \ref{fakeheeg}.  We call this operation a \emph{fake Heegaard stabilization}.  The effect of this operation is to perform a ``standard" stabilization between $H_1$ and $H_2$, as in Figure \ref{standstab}.
\end{defn}
\begin{rem}
\label{fakeheegremark}
A disk $D$ as required in the above definition always exists if we have just stabilized $H_1$.  Indeed, stabilizing changes $H_1$ by attaching a 1-handle to it, and a core disk of this one handle will satisfy the requirements.  Once we have found such a disk, we can use parallel copies of it to perform an arbitrary number of fake Heegaard stabilizations.
\end{rem}

\begin{figure}[ht]
\centering
    \begin{subfigure}[ht]{0.45\textwidth}
        \centering
        \includegraphics[height=1.6in]{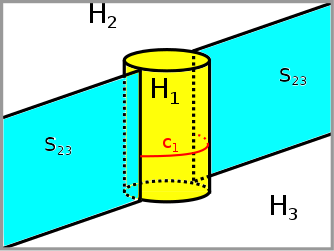}
\caption{}
    \end{subfigure}
    ~ 
    \begin{subfigure}[ht]{0.45\textwidth}
        \centering
        \includegraphics[height=1.6in]{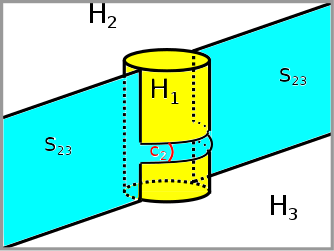}
\caption{}
    \end{subfigure}
\\
    \begin{subfigure}[ht]{0.45\textwidth}
        \centering
        \includegraphics[height=1.6in]{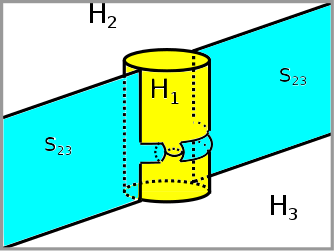}
\caption{}
    \end{subfigure}
    ~ 
    \begin{subfigure}[ht]{0.45\textwidth}
        \centering
        \includegraphics[height=1.6in]{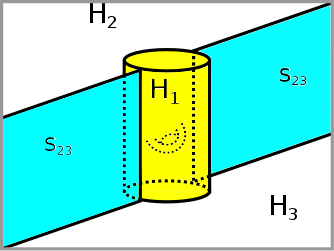}
\caption{}
    \end{subfigure}
\caption{We perform a fake Heegaard stabilization by stabilizing $H_2$ and then $H_1$.}
\label{fakeheeg}
\end{figure}

\begin{figure}[ht]
\centering
\includegraphics[scale=.35]{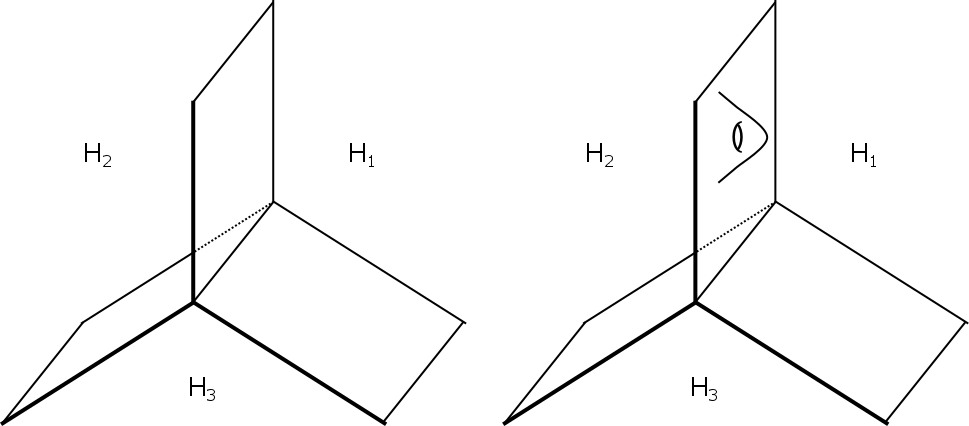}
\caption{A standard stabilization between $H_1$ and $H_2$ performed by adding to $H_2$ a regular neighborhood of an arc parallel into $S_{12}$, and removing the corresponding neighborhood from $H_1$.}
\label{standstab}
\end{figure}

Now we provide the proof of Theorem \ref{mainthm}, which we restate here for convenience.
\themainthm*
 Let  $(H_1,H_2,H_3;B)$ and $(H_1^*,H_2^*,H_3^*;B^*)$ be two trisections of a closed orientable 3-manifold $M$.  To avoid excessive notation, we use the same notation for both a trisection and the stabilizations that we obtain from that trisection.  The basic strategy is as follows:

\begin{enumerate}
\item Perform stabilizations until the quadruple $(h_1,h_2,h_3;b)$ is the same as $(h_1^*,h_2^*,h_3^*;b^*)$
\item Stabilize $H_1$ and $H_1^*$ until $H_2 \cup H_3$ and $H_2^* \cup H_3^*$ are handlebodies
\item Perform \emph{fake Heegaard stabilizations} until the Heegaard splittings $(H_1,H_2 \cup H_3;\del H_1)$ and $(H_1^*,H_2^* \cup H_3^*;\del H_1^*)$ are isotopic
\item Stabilize $H_3$ and $H_3^*$ until $S_{12}$ and $S_{12}^*$ are disks
\item Stabilize $H_2$ and $H_2^*$ until $S_{13}$ and $S_{13}^*$ are disks.
\end{enumerate}
After these steps, we will show that the resulting trisections are isotopic.  Since we will have started with two arbitrary trisections and stabilized both until we have isotopic trisections, the theorem follows.
\begin{proof}[Proof of Theorem \ref{mainthm}] \ 

\smallskip
\noindent\textit{Step 1.}\ \ 
First we stabilize so that the genera of the handlebodies in the two trisections are the same.  By applying Proposition \ref{balanceprop}, we may assume both trisections are balanced.  If $b > 2$ we stabilize the first trisection along an arc connecting two components of $B$, and then reapply Proposition $\ref{balanceprop}$.  Do the same for the second trisection if $b^* > 2$.  Then we have an $(h;b)$ and an $(h^*;b^*)$ balanced trisection where both $b,b^*$ are either 1 or 2.  If $h < h^*$, perform any stabilization on the first trisection, and then reapply Proposition \ref{balanceprop}, and repeat until $h = h^*$.  Do the same to the second trisection if $h^* < h$.  So we may assume that $h = h^*$ and, by the proof of Proposition \ref{balanceprop}, $b$ and $b^*$ must still be $\le 2$.  By Remark \ref{parityremark}, $b$ and $h$ must have opposite parities, so we see that $(h;b)$ is the same as $(h^*;b^*)$ as desired.  In future steps we perform stabilizations equally to both trisections so as to retain the property that both trisections have the same tuple $(h_1,h_2,h_3;b)$.

\smallskip
\noindent\textit{Step 2.}\ \ 
Choose a maximal nonseparating set of $h$ properly embedded arcs in $S_{23}$. Stabilizing along all arcs of this set results in a trisection where $S_{23}$ is a disk.  This means the complement of $H_1$ is the union of two handlebodies $H_2 \cup H_3$ glued along a disk in their boundaries, so it, too, is a handlebody.  Do the same thing to the other trisection.  Since we began this step with a balanced trisection with $h > 0$, $S_{23}$ was not a disk, so at least one stabilization was required in this step.  By Remark \ref{fakeheegremark}, both trisections now satisfy the necessary conditions to apply fake Heegaard stabilizations.

\smallskip
\noindent\textit{Step 3.}\ \ 
We now have that $(H_1,H_2 \cup H_3;\del H_1)$ and $(H_1^*,H_2^* \cup H_3^*;\del H_1^*)$ are Heegaard splittings of $M$.  By the Reidemeister-Singer theorem \cite{Reidemeister1}\cite{Singer1}, there exists a common Heegaard stabilization of these two Heegaard splittings.  The fake Heegaard stabilization operation affects the Heegaard splitting $(H_1,H_2 \cup H_3;\del H_1)$ just as Heegaard stabilization does.  Therefore, by repeatedly performing fake Heegaard stabilizations to both trisections we may assume that $(H_1,H_2 \cup H_3;\del H_1)$ and $(H_1^*,H_2^* \cup H_3^*;\del H_1^*)$ represent isotopic Heegaard splittings.  In particular, $H_1$ and $H_1^*$ are isotopic in $M$.

\smallskip
\noindent\textit{Step 4.}\ \ 
Choose a maximal set of nonseparating arcs properly embedded in $S_{12}$ and stabilize $H_3$ along all of them.  Do the same for the other trisection.  Since no stabilizations are performed on $H_1$ or $H_1^*$, $H_1$ and $H_1^*$ are still isotopic in $M$.  After doing this, $S_{12}$ and $S_{12}^*$ are disks.

\smallskip
\noindent\textit{Step 5.}\ \ 
Up to isotopy, we may now assume that $H_1 = H_1^*$ and $S_{12} = S_{12}^*$.  Since $S_ {12}$ is a disk, $H_1 \cup H_2$ can be obtained by attaching some 1-handles to $H_1$, with all attaching points occuring on $S_{12}$.  We show that if we were allowed to slide the ends of these handles along loops in $\del H_1$, we could arrange for them to be a set of small arcs each parallel $\mathrm{rel}~ \del$ into $S_{12}$.  Note that sliding the ends around freely will not necessarily correspond to an isotopy of the trisection because the ends may need to slide across $S_{13}$ to trivialize the handles.  See Figure \ref{h2as1handles}.  

To see that this is true, we define a Heegaard splitting of $H_2 \cup H_3$.  Let $W$ be the union of $H_2$ with a regular neighborhood of $\del(H_2 \cup H_3)$, so $V$ is a compression body.  Let $W$ be the complement of $V$, so $W$ is a slightly shrunken version of $H_3$, and is a handlebody. Then $V \cup W$ is a Heegaard splitting of $H_2 \cup H_3$ as desired.  We now use the fact that Heegaard splittings of handlebodies are standard, which follows from \cite{CassonGordon1}, \cite{Waldhausen1} and induction on genus.   Since $H_2 \cup H_3$ is a handlebody, the Heegaard splitting must be a stabilization of the standard one.  Therefore, $W$ must topologically be the union of $\del(H_2 \cup H_3)$ and some number of trivial 1-handles.  These 1-handles are therefore simultaneously parallel into $\del(H_2 \cup H_3)$.
\begin{figure}[ht]
\centering
\includegraphics[scale=.3]{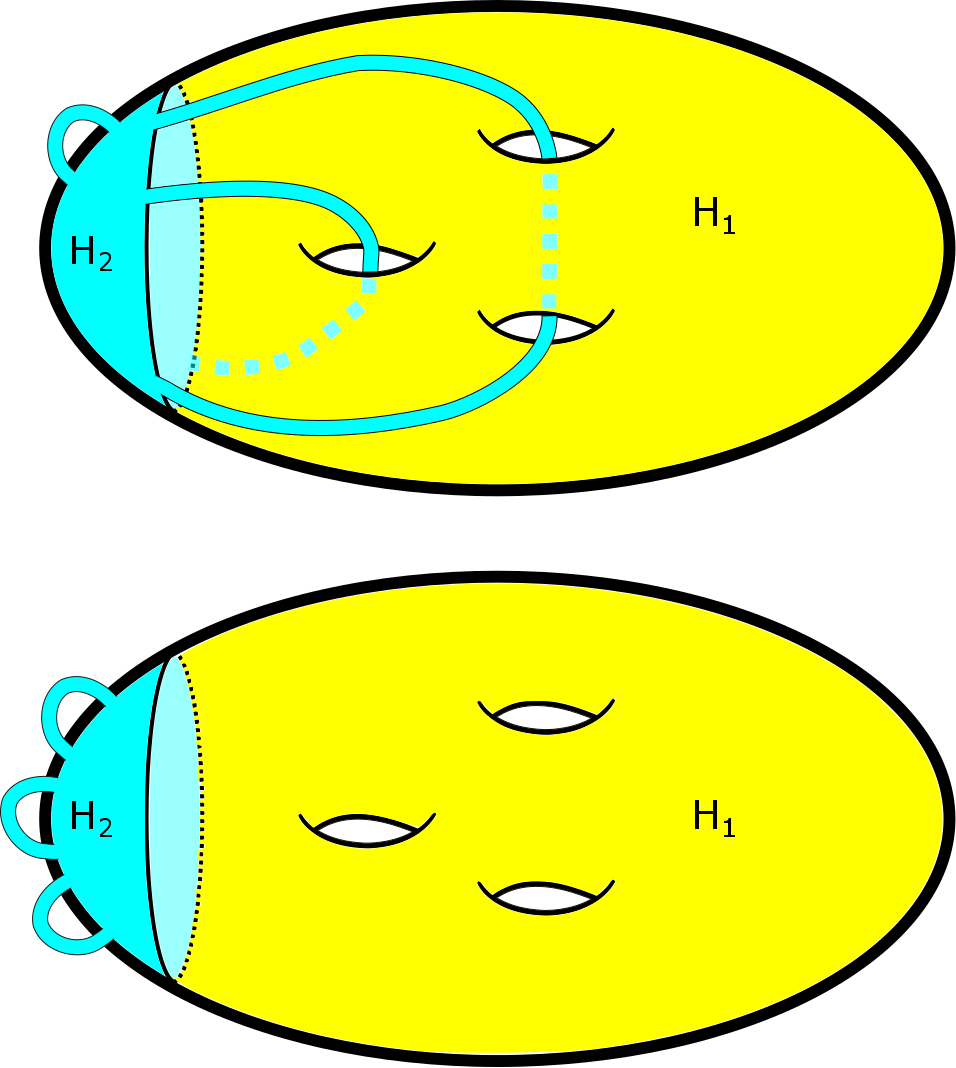}
\caption{After step 4, $H_2$ looks like a set of 1-handles attached to the thickened disk $N(S_{12})$.  The set of such 1-handles is simultaneously isotopic into $\del H_1$, so it is possible to slide the ends around  on $\del H_1$ to get to the lower picture where the set of handles is parallel into $S_{12}$.  However, performing such slides might require sliding the ends of the handles across $S_{13}$, which does not correspond to an isotopy of the trisection.}
\label{h2as1handles}
\end{figure}

To allow us to slide the ends of these 1-handles freely, stabilize $H_2$ as much as possible until $S_{13}$ is a disk.  After these stabilizations, $H_2$ consists of $N(\del H_1 - \{\mathrm{disk}\}) \cup \{\mathrm{trivial}$ 1-handles\}.  If we have done the same thing to the other trisection, we now know that there is an isotopy taking $H_1$ to $H_1^*$ and $H_2$ to $H_2^*$.  Note that if we hadn't performed step 1, it would be possible that $H_2$ had fewer or more of the trivial 1-handles than $H_2^*$.  The isotopy must also necessarily take $H_3$ to $H_3^*$, so the trisections are in fact isotopic.  Therefore, we have constructed a common stabilization of both initial trisections.  This concludes the proof. 

\end{proof}

\section{Trisections of $S^3$}

Recall that any Heegaard splitting of $S^3$ is a stabilization of the standard splitting into two balls \cite{Waldhausen1}.  One might hope for a similar result for trisections.  Since the genus 0 trisection cannot be stabilized, the simplest form of such a statement can be immediately ruled out.  However, one might still hope that there exists some finite list of low genus trisections such that any trisection of $S^3$ is obtained by stabilizing something in the list.  This turns out this too is impossible.  Specifically, 
\begin{prop}
There exists an infinite class of $(1,2,2;2)$ trisections of $S^3$ that are not stabilizations of any other trisection.  Therefore, there is no finite list of trisections of $S^3$ that can be stabilized to cover all possible trisections of $S^3$.
\end{prop}
\begin{proof}
First, we must investigate how to detect stabilized trisections.  Since a stabilization is performed by adding a neighborhood of an arc in $S_{jk}$ to $H_i$, it follows that we can detect destabilizations as follows:

\begin{defn}
Suppose $M$ is a closed oriented 3-manifold with trisection $(H_1,H_2,H_3;B)$.  A \emph{destabilizing disk} $D$ is an essential nonseparating disk properly embedded in some $H_i$ such that $\del D$ consists of a nonseparating arc of $S_{ij}$ and a nonseparating arc of $S_{ik}$.
\end{defn}

If there exists a destabilizing disk $D$, we can pinch $H_j$ and $H_k$ together across $D$, performing a compression on $H_i$.  Since $D$ is nonseparating, $H_i$ is still a handlebody, and $H_j$ and $H_k$ are unaffected topologically, so we still have a decomposition of $M$ into three handlebodies.  Since the arcs of $\del D$ in $S_{ij}$ and $S_{ik}$ were nonseparating, the two surfaces are still connected after the compression.  $S_{jk}$ has been affected by attaching a band, so it too is still connected.  Thus, all pairwise intersections are still connected.  Therefore the result of this operation is indeed still a trisection, and we call this trisection a \emph{destabilization} of the original trisection.   This destabilization operation is the reverse of a stabilizatiom.

Now we can demonstrate the class of examples.  Koda and Ozawa describe a class of knots whose exteriors contain an incompressible, boundary incompressible twice punctured genus-1 surface 
$\Sigma$ cutting the exterior into two handlebodies \cite{kodaozawa1}.  This surface intersects $N(K)$ in two toroidal curves with nonzero rational slope on $\del N(K)$.  Let $H_1$ be $\overline{N(K)}$ and let $H_2$ and $H_3$ be the two genus-2 handlebodies resulting from cutting the exterior of $K$ along $\Sigma$.  Let $B$ be $\Sigma \cap H_1$.  Then $(H_1,H_2,H_3;B)$ is a trisection of $S^3$.  Since the components of $B$ are toroidal curves on $\del H_1$, any disk properly embedded in $H_1$ must intersect each component of $B$ at least twice.  Therefore, no such disk can be a destabilizing disk.  Any destabilizing disk in $H_2$ or $H_3$ would contradict boundary incompressibility, so these also cannot exist.  It follows that this $(1,2,2;2)$-trisection is not a stabilization of any other trisection.  Since the class of knots provided by Koda and Ozawa is infinite, the proposition follows.
\end{proof}

\bibliographystyle{amsplain}
\bibliography{Trisections}{}

\providecommand{\bysame}{\leavevmode\hbox to3em{\hrulefill}\thinspace}
\providecommand{\MR}{\relax\ifhmode\unskip\space\fi MR }
\providecommand{\MRhref}[2]{%
  \href{http://www.ams.org/mathscinet-getitem?mr=#1}{#2}
}
\providecommand{\href}[2]{#2}
\begin{thebibliography}{10}

\bibitem{BachmanSchleimer1}
David Bachman and Saul Schleimer, \emph{Surface bundles versus {Heegaard}
  splittings}, Comm. Anal. Geom. \textbf{13} (2005), no.~5, 1--26.

\bibitem{CassonGordon1}
Andrew Casson and Cameron Gordon, \emph{Reducing {Heegaard} splittings},
  Topology appl. \textbf{27} (1987), 275--283.

\bibitem{CastroGayPinzon1}
Nickolas Castro, David Gay, and Juanita Pinz{\'o}n-Caicedo, \emph{Diagrams for
  relative trisections}, arXiv:1610.06373, 2016.

\bibitem{CoffeyRubinstein1}
James Coffey and Hyam Rubinstein, \emph{3-manifolds built from injective
  handlebodies}, Algebr. Geom. Topol. \textbf{17} (2017), 3213--3257.

\bibitem{GayKirby1}
David Gay and Robion Kirby, \emph{Trisecting 4-manifolds}, Geom. Topol.
  \textbf{20} (2016), 3097--3132.

\bibitem{GirouxGoodman1}
Emmanuel Giroux and Noah Goodman, \emph{On the stable equivalence of open books
  in three-manifolds}, Geom. Topol. \textbf{10} (2006), 97--114.

\bibitem{gomez1}
Jos{\'e}~Carlos Gomez-Larra{\~n}aga, \emph{3-manifolds which are unions of
  solid tori}, Manuscripta Math. \textbf{59} (1987), no.~3, 325--330.

\bibitem{gomezheilnunez1}
Jos{\'e}~Carlos Gomez-Larra{\~n}aga, Wolfgang Heil, and V{\'\i}ctor
  N{\'u}{\~n}ez, \emph{Stiefel-whitney surfaces and decompositions of
  3-manifolds into handlebodies}, Topology appl. \textbf{60} (1994), no.~3,
  267--280.

\bibitem{gomezheilnunez2}
\bysame, \emph{Stiefel-whitney surfaces and the tri-genus of non-orientable
  3-manifolds}, Manuscripta Math. \textbf{100} (1999), no.~4, 405--422.

\bibitem{Haken1}
Wolfgang Haken, \emph{Some results on surfaces in 3-manifolds}, Studies in
  {M}odern {T}opology, Math. Assoc. Amer. (distributed by Prentice-Hall,
  Englewood Cliffs, N.J.), 1968, pp.~39--98.

\bibitem{Hempel1}
John Hempel, \emph{3-manifolds as viewed from the curve complex}, Topology
  \textbf{40} (2001), no.~3, 631--657.

\bibitem{kodaozawa1}
Yuya Koda and Makoto Ozawa, \emph{Essential surfaces of non-negative {E}uler
  characteristic in genus two handlebody exteriors}, Trans. Amer. Math. Soc.
  \textbf{367} (2015), 2875--2904.

\bibitem{Reidemeister1}
Kurt Reidemeister, \emph{Zur dreidimensionalen topologie}, Abh. Math. Sem.
  Univ. Hamburg \textbf{11} (1933), 189--194.

\bibitem{ScharlemannTomova1}
Martin Scharlemann and Maggy Tomova, \emph{Alternate {Heegaard} genus bounds
  distance}, Geom. Topol. \textbf{10} (2006), 593--617.

\bibitem{Schultens1}
Jennifer Schultens, \emph{The classification of {Heegaard} splittings for
  (closed orientable surface){ $\times ~S^1$}}, {Proc. London Math. Soc.}
  \textbf{67} (1993), 401--487.

\bibitem{Singer1}
James Singer, \emph{Three-dimensional manifolds and their {Heegaard} diagrams},
  Trans. Amer. Math. Soc. \textbf{35} (1933), 88--111.

\bibitem{Waldhausen1}
Friedhelm Waldhausen, \emph{Heegaard-zerlegungen der 3-sph{\"a}re}, Topology
  \textbf{7} (1968), 195--203.

\end{thebibliography}

%

%

\end{document}